\documentclass[a4paper,12pt]{amsart}
\calclayout
\usepackage{amssymb}
\usepackage{amsmath}
\usepackage{amsfonts}
\usepackage[all]{xy}
\usepackage{mathabx}

\newtheorem{thm}{Theorem}[section]
\newtheorem{cor}[thm]{Corollary}
\newtheorem{lem}[thm]{Lemma}
\newtheorem{prop}[thm]{Proposition}
\theoremstyle{definition}
\newtheorem{defn}[thm]{Definition}
\theoremstyle{remark}

\numberwithin{equation}{section}

\begin{document}

\title[The predual of space of decomposable maps]{The predual of the space of decomposable maps from a $C^*$-algebra into a von Neumann algebra}

\author{Kyung Hoon Han}

\address{Department of Mathematics, The University of Suwon, Gyeonggi-do 445-743, Korea}

\email{kyunghoon.han@gmail.com}



\thanks{This work was supported by the National Research Foundation of Korea Grant funded by the Korean Government (2012R1A1A1012190)}

\date{}

\dedicatory{}

\commby{}


\begin{abstract}
For a $C^*$-algebra $\mathcal A$ and a von Neumann algebra $\mathcal R$, we describe the predual of space $D(\mathcal A,\mathcal R)$ of decomposable maps from $\mathcal A$ into $\mathcal R$ equipped with decomposable norm. This predual is found to be the matrix regular operator space structure on $\mathcal A \otimes \mathcal R_*$ with a certain universal property. Its matrix norms are the largest and its positive cones on each matrix level are the smallest among all possible matrix regular operator space structures on $\mathcal A \otimes \mathcal R_*$ under the two natural restrictions: (1) $\|x \otimes y\| \le \|x\| \|y\|$ for $x\in M_k(\mathcal A), y \in M_l(\mathcal R_*)$ and (2) $v \otimes w$ is positive if $v \in M_k(\mathcal A)^+$ and $w \in M_l(\mathcal R_*)^+$.
\end{abstract}

\maketitle

\section{Introduction}
In \cite{H}, Haagerup introduced the concept of decomposable maps between $C^*$-algebras and their decomposable norms in order to prove the converse of the Wittstock decomposition theorem. According to the Wittstock decomposition theorem, every completely bounded map of $C^*$-algebra $T : \mathcal A \to \mathcal B$ is a linear combination of completely positive maps if range $\mathcal B$ is an injective $C^*$-algebra. Conversely, Haagerup proved that if range $\mathcal B$ is a von Neumann algebra, and for all $C^*$-algebras $\mathcal A$, if every completely bounded map $T : \mathcal A \to \mathcal B$ is a linear combination of completely positive maps, then the von Neumann algebra $\mathcal B$ is necessarily injective. Motivated by Paulsen's simple proof on the Wittstock decomposition theorem \cite{Pa1}, decomposable norms have been introduced, which are different from completely bounded norms and well fitted to decomposition into completely positive maps. The Banach space of decomposable maps from $\mathcal A$ into $\mathcal B$ equipped with decomposable norm is denoted by $D(\mathcal A, \mathcal B)$.

Let $\mathcal A$ be a $C^*$-algebra and $\mathcal R$ a von Neumann algebra. Recall that the dual of the projective tensor product $\mathcal A \otimes_\gamma \mathcal R_*$ is the space $B(\mathcal A, \mathcal R)$ of bounded linear maps from $\mathcal A$ into $\mathcal R$ in a canonical manner, and that the dual of the operator space projective tensor product $\mathcal A \otimes_\wedge \mathcal R_*$ is the space $CB(\mathcal A, \mathcal R)$ of completely bounded linear maps. It is natural to ask what is the predual of the space $D(\mathcal A, \mathcal R)$ of decomposable maps from $\mathcal A$ into $\mathcal R$. The operator space predual of $D(\mathcal A, \mathcal R)$ had been already found by Le Merdy and Magajna \cite[Theorem 6.1]{LMM}. In this paper, we give the explicit and intrinsic description of the matrix regular operator space predual of $D(\mathcal A, \mathcal R)$. This predual is found to be the matrix regular operator space structure on $\mathcal A \otimes \mathcal R_*$ with a certain universal property. Its matrix norms are the largest and its positive cones on each matrix level are the smallest among all possible matrix regular operator space structures on $\mathcal A \otimes \mathcal R_*$ under the two natural restrictions: (1) $\|v \otimes w\| \le \|v\| \|w\|$ for $v \in M_k(\mathcal A), w \in M_l(\mathcal R_*)$ and (2) $v \otimes w$ is positive if $v \in M_k(\mathcal A)^+$ and $w \in M_l(\mathcal R_*)^+$.

We also consider the opposite situation, the matrix regular operator space structure on $\mathcal A^* \otimes \mathcal B$ whose matrix norms are the smallest and the positive cones at each matrix level are the largest among all possible matrix regular operator space structures on $\mathcal A^* \otimes \mathcal B$ under the natural restriction: $\varphi \otimes \psi : \mathcal A^* \otimes \mathcal B \to M_{kl}$ is completely positive and completely contractive (c.c.p.) for all c.c.p. maps $\varphi : \mathcal A^* \to M_k$ and $\psi : \mathcal B \to M_l$. We will show that the canonical embedding $\mathcal A^* \otimes \mathcal B \hookrightarrow D(\mathcal A, \mathcal B)$ is isometric and order isomorphic with respect to such a matrix regular operator space structure on $\mathcal A^* \otimes \mathcal B$. Therefore, its norm is consistent with Pisier's delta norm \cite{JLM,Pi}.

The methods of defining two extremal tensor products of matrix regular operator spaces can be regarded as a combination of those for operator spaces \cite{BP, ER1} and those for operator systems \cite{KPTT}.

\section{Preliminaries}

A linear mapping of $C^*$-algebras is called decomposable if it is a linear combination of completely positive maps from $\mathcal A$ into $\mathcal B$. A linear map $T : \mathcal A \to \mathcal B$ is decomposable if and only if there exist two completely positive maps $S_1, S_2 : \mathcal A \to \mathcal B$ for which the linear map $$\Phi : \mathcal A \to M_2(\mathcal B) : a \mapsto \begin{pmatrix} S_1(a) & T(a) \\ T^*(a) & S_2(a) \end{pmatrix}$$ is completely positive. The adjoint $T^* : \mathcal A \to \mathcal B$ is defined as $T^*(a)=T(a^*)^*$. Given a decomposable map $T : \mathcal A \to \mathcal B$, Haagerup defined its decomposable norm $\|T\|_{dec}$ by $$\|T\|_{dec} = \inf \{ \max \{ \|S_1\|, \|S_2\|\} : \begin{pmatrix} S_1 & T \\ T^* & S_2 \end{pmatrix} \ge_{cp} 0 \}.$$ The Banach space of decomposable maps from $\mathcal A$ into $\mathcal B$ equipped with decomposable norms $\|\cdot\|_{dec}$ is denoted by $D(\mathcal A, \mathcal B)$.

Let $E$ be an operator space and $\mathcal B$ a $C^*$-algebra. Pisier defined the delta norm on $E \otimes \mathcal B$ by $$\delta(z) = \sup \{ \|(\sigma \cdot \pi)(z)\|\}, \qquad z \in E \otimes B,$$ where the supremum is taken over all complete contractions $\sigma : E \to B(H)$ and $*$-representations $\pi : \mathcal B \to B(H)$ with commuting ranges. Based on Pisier's result \cite[Chapter 12]{Pi}, Junge and Le Merdy proved that the canonical embedding $\mathcal A^* \otimes \mathcal B \hookrightarrow D(\mathcal A, \mathcal B)$ is isometric with respect to the delta norm \cite{JLM}.

A complex vector space $V$ is matrix ordered if
\begin{enumerate}
\item $V$ is a $*$-vector space (hence so is $M_n(V)$ for all $n
\ge 1$), \item each $M_n(V), n \ge 1$, is partially ordered by a
(not necessarily proper) cone $M_n(V)^+ \subset  M_n(V)_{sa}$ ,
and \item if $\alpha \in M_{m,n}$, then $\alpha^* M_m(V)^+ \alpha
\subset M_n(V)^+$.
\end{enumerate}

An operator space $V$ is called a matrix ordered operator space
iff $V$ is a matrix ordered vector space and for every $n \in
\mathbb N$,
\begin{enumerate}
\item the $*$-operation is an isometry on $M_n(V)$, and \item the
cones $M_n(V)^+$ are closed.
\end{enumerate}

For a matrix ordered operator space $V$ and its dual space $V^*$,
the positive cone on $M_n(V^*)$ for each $n \in \mathbb N$ is
defined by
$$M_n(V^*)^+=CB(V, M_n) \cap CP(V, M_n).$$ Then the operator space dual
$V^*$ with this positive cone is a matrix ordered operator space
\cite[Corollary 3.2]{S}.

For a matrix ordered operator space $V$, we say
that $V$ is a matrix regular operator space if for each $n \in
\mathbb N$ and for all $v \in M_n(V)_{sa}$
\begin{enumerate}
\item $u \in M_n(V)^+$ and $-u \le v \le u$ imply that $\|v\|_n
\le \|u\|_n$, and \item $\|v\|_n<1$ implies that there exists $u
\in M_n(V)^+$ such that $\|u\|_n<1$ and $-u \le v \le u$.
\end{enumerate}
Due to condition (1), it is easily seen that the positive cone of
a matrix regular operator space is always proper. In \cite{S}, the norms of matrix regular operator spaces are assumed to be complete. However, for simplicity, we do not make this assumption because we do not use \cite[Theorem 4.10]{S}, wherein this completeness is indispensable.

A matrix regular operator space can be described alternatively. A matrix ordered operator space $V$ is matrix regular if and only if the following condition holds: for all $x \in M_n(V), \|x\|_n<1$ if and only if there exist $a, d \in M_n(V)^+, \|a\|_n<1$, and $\|d\|_n<1$, such that $\begin{pmatrix} a & x \\ x^* & d \end{pmatrix} \in M_{2n}(V)^+$ \cite[Theorem 3.4]{S}.

The dual spaces of matrix regular operator spaces are matrix regular as well \cite[Corollary 4.7]{S}. The class of matrix regular operator spaces contains $C^*$-algebras and their duals, preduals of von Neumann algebras, and the Schatten class $\mathcal S_p$.

\section{Tensor products of matrix regular operator spaces}

In this section, we construct two extremal tensor products of matrix regular operator spaces. The methods to do so can be considered as a combination of those for operator spaces \cite{BP,ER1} and those for operator systems \cite{KPTT}. The basic observation lies in the fact that the matrix norms of matrix regular operator spaces are determined by the positive cones at each matrix level and the matrix norms on them, $$\|x\|_{M_n(V)} = \inf \{ \max \{ \|a\|_{M_n(V)}, \|d\|_{M_n(V)} \} : \begin{pmatrix} a & x \\ x^* & d \end{pmatrix} \in M_{2n}(V)^+\}.$$ Henceforth, we first describe the positive cones at each matrix level, and then assign them scalar values. Next, we extend the assignment on the entire space using the above formula. Finally, we prove that they are actually matrix regular operator spaces.

In often, we abbreviate completely positive maps and completely positive completely contractive maps as c.p. maps and c.c.p. maps, respectively.

\begin{defn}
Suppose that $V$ and $W$ are matrix regular operator spaces. \begin{enumerate} \item We let $M_n(V \otimes_{\delta} W)_+ := \{ z \in M_n (V \otimes W) : [(\varphi \otimes \psi) (z_{ij})] \in M_{nkl}^+$ for all continuous c.p. maps $\varphi : V \to M_k, \psi : W \to M_l, k,l \in \mathbb N \}$. \item For an element $z$ in $M_n(V \otimes_{\delta} W)_+$, let $$|z|_{\delta} := \sup \{ \| [\varphi \otimes \psi (z_{ij})]\| : \text{c.c.p. maps}\ \varphi : V \to M_k, \psi : W \to M_l, \quad k,l \in {\mathbb N} \}.$$ \item For a general element $z$ in $M_n(V \otimes W)$, let $$\|z\|_{\delta} := \inf \{ \max \{ |u|_{\delta}, |u'|_{\delta} \} : \begin{pmatrix} u & z \\ z^* & u' \end{pmatrix} \in M_{2n}(V \otimes_{\delta} W)_+ \}.$$ \end{enumerate}
\end{defn}

Actually, we will show that $|z|_{\delta}=\|z\|_\delta = \|z\|_\vee$ for $z \in M_n(V \otimes_\delta W)_+$ in Lemma \ref{positive}. The notation $\delta$ was chosen owing to its consistency with Pisier's delta norm, which will be proved in Corollary \ref{pisier}.

\begin{defn}
Suppose that $V$ and $W$ are matrix regular operator spaces. \begin{enumerate} \item We let $$M_n(V \otimes_{\Delta} W)_+ := \{ \alpha (v \otimes w) \alpha^* \in M_n (V \otimes W) : v \in M_k(V)^+, w \in M_l(W)^+, \alpha \in M_{n,kl}, k,l \in {\mathbb N} \}.$$ \item For an element $z$ in $M_n(V \otimes_{\Delta} W)_+$, let $$|z|_{\Delta} := \inf \{ \|\alpha\|^2 \|v\| \|w\| : z=\alpha (v \otimes w) \alpha^*, \alpha \in M_{n,kl}, v \in M_k(V)^+, w \in M_l (W)^+ \}.$$ \item For a general element $z$ in $M_n(V \otimes W)$, let $$\|z\|_{\Delta} := \inf \{ \max \{ |u|_{\Delta}, |u'|_{\Delta} \} : \begin{pmatrix} u & z \\ z^* & u' \end{pmatrix} \in M_{2n}(V \otimes_{\Delta} W)_+ \}.$$ \end{enumerate}
\end{defn}

When proving Theorem \ref{Delta}, we will see that $\|z\|_\Delta \le |z|_\Delta$ for $z \in M_n(V \otimes_\Delta W)_+$. The immediate question whether the equality holds is not answered at the time of this writing. In the above definitions, the set over which an infimum is taken is nonempty according to Propositions \ref{nonempty} and \ref{basic}(1).

\begin{prop}\label{nonempty}
Suppose that $V$ and $W$ are matrix regular operator spaces and $z$ is an element in $M_n(V \otimes W)$. There exist elements $u_1, u_2$ in $M_n(V \otimes_{\Delta} W)_+$ such that $\begin{pmatrix} u_1 & z \\ z^* & u_2 \end{pmatrix} \in M_{2n}(V \otimes_{\Delta} W)_+$.
\end{prop}

\begin{proof}
Suppose that $z=\alpha (v \otimes w) \beta^*$ for $v \in M_k(V), w \in M_l(W)$ and $\alpha, \beta \in M_{n,kl}$. By the matrix regularity of $V$ and $W$, there exist $v_1, v_2 \in M_k(V)^+$ and $w_1, w_2 \in M_l(W)^+$ such that $$\begin{pmatrix} v_1 & v \\ v^* & v_2 \end{pmatrix} \in M_{2k}(V)^+ \qquad \text{and} \qquad \begin{pmatrix} w_1 & w \\ w^* & w_2 \end{pmatrix} \in M_{2l}(W)^+.$$ Because the matrix $$\begin{pmatrix} \alpha (v_1 \otimes w_1) \alpha^* & \alpha (v \otimes w) \beta^* \\ \beta (v^* \otimes w^*) \alpha^* & \beta (v_2 \otimes w_2) \beta^* \end{pmatrix} = \begin{pmatrix} \alpha&0&0&0 \\ 0&0&0&\beta \end{pmatrix} \begin{pmatrix} v_1 & v \\ v^* & v_2 \end{pmatrix} \otimes \begin{pmatrix} w_1 & w \\ w^* & w_2 \end{pmatrix} \begin{pmatrix} \alpha^* &0 \\ 0&0 \\ 0&0 \\ 0&\beta^* \end{pmatrix}.$$
belongs to $M_{2n}(V \otimes_{\Delta} W)_+ $, we can take $$u_1= \alpha (v_1 \otimes w_1) \alpha^* \qquad \text{and} \qquad u_2 =\beta (v_2 \otimes w_2) \beta^*.$$

\end{proof}

\begin{prop}
Suppose that $V$ and $W$ are matrix regular operator spaces. Then we have
\begin{enumerate}
\item $\|z_1+z_2\|_{\Delta} \le \|z_1\|_{\Delta}+\|z_2\|_{\Delta}, \qquad z_1, z_2 \in M_n(V \otimes W)$,
\item $\|z_1+z_2\|_{\delta} \le \|z_1\|_{\delta}+\|z_2\|_{\delta}, \qquad z_1, z_2 \in M_n(V \otimes W)$.
\end{enumerate}
\end{prop}

\begin{proof}
(1) First, we suppose that elements $z_1, z_2$ belong to $M_n(V \otimes_{\Delta} W)_+$. There exist $v_i \in M_k(V)^+ , w_i \in M_l(W)^+ , \alpha_i \in M_{n,kl}$ ($i=1,2$) such that $$z_i = \alpha_i (v_i \otimes w_i) \alpha_i^* \qquad \text{and} \qquad  \|\alpha_i\|^2 \|v_i\| \|w_i\|<|z_i|_{\Delta}+\varepsilon \qquad (i=1,2).$$ We may assume that $\|v_1\| = \|w_1\| = \|v_2\| = \|w_2\|=1$ by homogeneity. Since $$\begin{aligned} z_1+z_2 & = \alpha_1 v_1 \otimes w_1 \alpha_1^* + \alpha_2 v_2 \otimes w_2 \alpha_2^* \\ & = \begin{pmatrix} \alpha_1 &0&0& \alpha_2 \end{pmatrix} \begin{pmatrix} v_1&0 \\ 0&v_2 \end{pmatrix} \otimes \begin{pmatrix} w_1&0 \\ 0&w_2 \end{pmatrix} \begin{pmatrix} \alpha_1^* \\ 0 \\ 0 \\ \alpha_2^* \end{pmatrix}, \end{aligned}$$ we have $$\begin{aligned} |z_1+z_2|_{\Delta} & \le \|\begin{pmatrix} \alpha_1 &0&0& \alpha_2 \end{pmatrix} \|^2 \| \begin{pmatrix} v_1&0 \\ 0&v_2 \end{pmatrix}\| \|\begin{pmatrix} w_1&0 \\ 0&w_2 \end{pmatrix} \| \\ & = \|\alpha_1 \alpha_1^*+\alpha_2 \alpha_2^*\| \\ & \le \|\alpha_1\|^2+\|\alpha_2\|^2 \\ & < |z_1|_{\Delta}+|z_2|_{\Delta} +2 \varepsilon,\end{aligned}$$ thus, $|z_1+z_2|_{\Delta} \le |z_1|_{\Delta}+|z_2|_{\Delta}$.

Next, we consider general elements $z_1, z_2$ in $M_n(V \otimes W)$. There exist elements $u_i, u_i'$ in $M_n(V \otimes_{\Delta} W)_+$ ($i=1,2$) such that $$\begin{pmatrix} u_i & z_i \\ z_i^* & u_i' \end{pmatrix} \in M_{2n}(V \otimes_{\Delta} W)_+ \qquad \text{and} \qquad |u_i|_{\Delta}, |u_i'|_{\Delta} < \|z_i\|_{\Delta}+\varepsilon.$$ From $$\begin{pmatrix} u_1+u_2 & z_1+z_2 \\ z_1^*+z_2^* & u_1'+u_2' \end{pmatrix} \in M_{2n}(V \otimes_{\Delta} W)_+,$$ it follows that $$\begin{aligned} \|z_1+z_2\|_{\Delta} & \le \max \{ |u_1+u_2|_{\Delta}, |u_1'+u_2'|_{\Delta} \} \\ & \le \max \{ |u_1|_{\Delta}+|u_2|_{\Delta}, |u_1'|_{\Delta}+|u_2'|_{\Delta} \} \\ & <\|z_1\|_{\Delta}+\|z_2\|_{\Delta}+2\varepsilon. \end{aligned}$$

(2) First, we suppose that elements $z_1, z_2$ belong to $M_n(V \otimes_{\delta} W)_+$. For c.c.p. maps $\varphi : V \to M_k$ and $\psi : W \to M_l$, we have $$\|(\varphi \otimes \psi)_n(z_1+z_2)\| \le \|(\varphi \otimes \psi)_n(z_1)\|+\|(\varphi \otimes \psi)_n(z_2)\| \le |z_1|_\delta + |z_2|_\delta,$$ thus, $|z_1+z_2|_{\delta} \le |z_1|_{\delta}+|z_2|_{\delta}$. The remaining proof is similar to (1).
\end{proof}

\begin{prop}\label{basic}
Suppose that $V$ and $W$ are matrix regular operator spaces.
\begin{enumerate}
\item $M_n(V \otimes_{\Delta} W)_+$ is a subcone of the proper cone $M_n(V \otimes_{\delta} W)_+$. \item $\| z \|_{\delta} \le \|z\|_{\Delta},\qquad z \in M_n(V \otimes W)$. \item $\| \cdot \|_{\delta}$ and $\| \cdot \|_{\Delta}$ are norms on $M_n(V \otimes W)$.
\end{enumerate}
\end{prop}

\begin{proof}
(1) Suppose that $z \in M_n(V \otimes_\delta W)_+ \cap -M_n(V \otimes_\delta W)_+$. For continuous c.p. maps $\varphi : V \to M_k$ and $\psi : W \to M_l$, we have $$(\varphi \otimes \psi)_n(z) \in M_{nkl}^+ \cap - M_{nkl}^+ = \{ 0 \}.$$ The dual spaces of matrix regular operator spaces are also matrix regular \cite[Corollary 4.7]{S}. It implies that every completely bounded linear map from a matrix regular operator space into a matrix algebra is a linear combination of completely positive maps. Therefore, we have $\|z\|_{M_n(V \widecheck \otimes W)}=0$, which implies that $z=0$. The cone $M_n(V \otimes_{\delta} W)_+$ is proper. Let $$z = \alpha (v \otimes w) \alpha^* \in M_n(V \otimes_\Delta W)_+, \qquad v \in M_p(V)^+, w \in M_q(W)^+, \alpha \in M_{n,pq}.$$ Since $$(\varphi \otimes \psi)_n (z) = \alpha (\varphi_p(v) \otimes \psi_q(w)) \alpha^* \in M_{nkl}^+,$$ $M_n(V \otimes_{\Delta} W)_+$ is a subcone of the proper cone $M_n(V \otimes_{\delta} W)_+$.

(2) Suppose that $\varphi : V \to M_k$ and $\psi : W \to M_l$ are c.c.p. maps. Since $$\|(\varphi \otimes \psi)_n (\alpha (v \otimes w) \alpha^*)\| \le \|\alpha\|^2\|\varphi_p(v)\| \|\psi_q(w)\| \le \|\alpha\|^2 \|v\| \|w\|,$$ we have $$|u|_{\delta} \le |u|_{\Delta}, \qquad u \in M_n(V \otimes_{\Delta} W)_+.$$ For a general element $z$ in $M_n(V \otimes W)$, there exist elements $u, u'$ in $M_n(V \otimes W)$ such that $$|u|_{\Delta}, |u'|_{\Delta} < \|z\|_{\Delta}+\varepsilon \qquad \text{and} \qquad \begin{pmatrix} u&z \\ z^*&u' \end{pmatrix} \in M_{2n}(V \otimes_\Delta W)_+.$$ Since $|u|_{\delta} \le |u|_{\Delta}$, $|u'|_{\delta} \le |u'|_{\Delta}$ and $M_{2n}(V \otimes_{\Delta} W)_+ \subset M_{2n}(V \otimes_{\delta} W)_+$, we see that $$\|z\|_{\delta} \le \max \{ |u|_\delta, |u'|_\delta \} \le \max \{ |u|_\Delta, |u'|_\Delta \} < \|z\|_{\Delta}+\varepsilon.$$

(3) Suppose that $\|z\|_{\delta}=0$ for $z \in M_n(V \otimes W)$. There exist $u, u' \in M_n(V \otimes_\delta W)_+$ such that $$\begin{pmatrix} u&z \\ z^*&u' \end{pmatrix} \in M_{2n}(V \otimes_\delta W)_+ \qquad \text{and} \qquad |u|_{\delta}, |u'|_{\delta} < \varepsilon.$$ For c.c.p. maps $\varphi : V \to M_k$ and $\psi : W \to M_l$, we have $$\begin{pmatrix} (\varphi \otimes \psi)_n(u) & (\varphi \otimes \psi)_n(z) \\ (\varphi \otimes \psi)_n(z^*) &  (\varphi \otimes \psi)_n(u') \end{pmatrix} \in M_{2nkl}^+.$$ It follows that  $$\|(\varphi \otimes \psi)_n(z)\| \le \max \{ \|(\varphi \otimes \psi)_n(u)\|, \|(\varphi \otimes \psi)_n(u')\| \} \le \max \{ |u|_{\delta}, |u'|_{\delta} \} < \varepsilon,$$ thus $(\varphi \otimes \psi)_n(z)=0$. Since every completely bounded linear map from a matrix regular operator space into a matrix algebra is a linear combination of completely positive maps, we have $\|z\|_{M_n(V \widecheck \otimes W)}=0$, which implies that $z=0$. By (2), $\|z\|_{\Delta}=0$ also implies that $z=0$.
\end{proof}

Recall that the positive cones of matrix ordered operator spaces are closed. Therefore, a closure process is required. We denote the closure of $M_n(V \otimes_\delta W)_+$ with respect to norm $\|\cdot\|_{M_n(V \otimes_\delta W)}$ by $M_n(V \otimes_{\delta} W)^+$. We also denote the closure of $M_n(V \otimes_\Delta W)_+$ with respect to norm $\|\cdot\|_{M_n(V \otimes_\Delta W)}$ by $M_n(V \otimes_{\Delta} W)^+$ .

\begin{prop}
Suppose that $V$ and $W$ are matrix regular operator spaces.
\begin{enumerate}
\item The involution is an isometry on $M_n(V \otimes_\Delta W)$ and $M_n(V \otimes_\delta W)$, respectively.
\item Elements in $M_n(V \otimes_\Delta W)^+$ and $M_n(V \otimes_\delta W)^+$ are self-adjoint.
\item $M_n(V \otimes_\delta W)^+ = M_n(V \otimes_\delta W)_+$.
\item $M_n(V \otimes_\Delta W)^+$ is a subcone of the proper cone $M_n(V \otimes_\delta W)^+$.
\end{enumerate}
\end{prop}

\begin{proof}
(1) From $$\begin{pmatrix} 0&1 \\ 1&0 \end{pmatrix} \begin{pmatrix} u&z \\ z^*&u' \end{pmatrix} \begin{pmatrix} 0&1 \\ 1&0 \end{pmatrix} = \begin{pmatrix} u'&z^* \\ z&u \end{pmatrix},$$ we see that $\|z\|_\Delta = \|z^*\|_\Delta$ and $\|z\|_\delta = \|z^*\|_\delta$ for $z \in M_n(V \otimes W)$.

(2) It is obvious that elements in $M_n(V \otimes_\Delta W)_+$ are self-adjoint. If $z \in M_n(V \otimes_\delta W)_+$ and $\varphi : V \to M_k, \psi : W \to M_l$ are c.c.p. maps, then we have $(\varphi \otimes \psi)_n(z)=(\varphi \otimes \psi)_n(z)^*=(\varphi \otimes \psi)_n(z^*)$. Because every completely bounded map from a matrix regular operator space into a matrix algebra is a linear combination of completely positive maps, $\|z-z^*\|_{M_n(V \otimes_\vee W)}=0$, and hence $z=z^*$. By (1), elements in their closures are also self-adjoint.

(3) Suppose that $z_j \in M_n(V \otimes_\delta W)_+$ converges to $z$ with respect to norm $\|\cdot\|_\delta$. There exist $u_j, u_j'$ in $M_n(V \otimes_\delta W)_+$ such that $$\begin{pmatrix} u_j & z-z_j \\ z^*-z_j^* & u_j' \end{pmatrix} \in M_{2n}(V \otimes_\delta W)_+ \qquad \text{and} \qquad \lim_{j \to \infty} |u_j|_\delta = 0 =\lim_{j \to \infty} |u_j'|_\delta.$$ It follows that $$\|(\varphi \otimes \psi)_n(z-z_j)\| \le \max \{ \|(\varphi \otimes \psi)_n (u_j)\|, \|(\varphi \otimes \psi)_n (u_j')\| \} \le \max \{ |u_j|_\delta, |u_j'|_\delta \} \to 0$$ for all c.c.p. maps $\varphi : V \to M_k, \psi : W \to M_l$, so, $(\varphi \otimes \psi)_n(z) \ge 0$.

(4) It follows from Proposition \ref{basic} (1) and (2).

\end{proof}

\begin{thm}\label{Delta}
Suppose that $V$ and $W$ are matrix regular operator spaces. Then, $$(V \otimes W, \{ \|\cdot\|_{M_n(V \otimes_{\Delta} W)} \}_{n \in \mathbb N}, \{ M_n(V \otimes_{\Delta} W)^+) \}_{n \in \mathbb N})$$ is a matrix regular operator space with a subcross matrix norm.
\end{thm}

\begin{proof}
Let $z_i \in M_{n_i}(V \otimes W)$ ($i=1,2$). There exist $u_i, u_i' \in M_{n_i}(V \otimes_\Delta W)_+$ such that $$\begin{pmatrix} u_i & z_i \\ z_i^* & u_i' \end{pmatrix} \in M_{2n_i}(V \otimes_\Delta W)_+ \qquad \text{and} \qquad |u_i|_{\Delta}, |u_i'|_{\Delta} < \|z_i\|_{\Delta}+\varepsilon \qquad (i=1,2).$$ Let $$u_i =\alpha_i v_i \otimes w_i \alpha_i^*,\qquad \alpha_i \in M_{n_i,k_i l_i}, v_i \in M_{k_i}(V)^+, w_i \in M_{l_i}(W)^+$$ with $\|\alpha_i\|^2\|v_i\|\|w_i\|<|u_i|_\Delta+\varepsilon'$. We may assume that $\|v_i\|=\|w_i\|=1$ by homogeneity. Since $$\begin{pmatrix} u_1 & 0 \\ 0 & u_2 \end{pmatrix} = \begin{pmatrix} \alpha_1 &0&0&0 \\ 0&0&0& \alpha_2 \end{pmatrix} \begin{pmatrix} v_1 & 0 \\ 0 & v_2 \end{pmatrix} \otimes \begin{pmatrix} w_1 & 0 \\ 0 & w_2 \end{pmatrix} \begin{pmatrix} \alpha_1^* & 0 \\ 0&0 \\ 0&0 \\ 0& \alpha_2^* \end{pmatrix},$$
we have $$|\begin{pmatrix} u_1 & 0 \\ 0 & u_2 \end{pmatrix}|_\Delta \le \max \{ \|\alpha_1\|^2, \|\alpha_2\|^2 \} < \max\{|u_1|_\Delta, |u_2|_\Delta \} +\varepsilon'.$$
Consequently, $$|\begin{pmatrix} u_1 & 0 \\ 0 & u_2 \end{pmatrix}|_\Delta \le \max \{ |u_1|_\Delta,|u_2|_\Delta \}.$$
From $$\begin{pmatrix} u_1& & z_1 & \\ &u_2&&z_2 \\ z_1^*&& u_1'& \\ &z_2^*&&u_2' \end{pmatrix} = \begin{pmatrix} 1&&&\\&&1&\\&1&&\\&&&1 \end{pmatrix} \begin{pmatrix} u_1&z_1&&\\z_1^*&u_1'&&\\&&u_2&z_2\\&&z_2^*&u_2'\end{pmatrix} \begin{pmatrix} 1&&&\\&&1&\\&1&&\\&&&1 \end{pmatrix} \in M_{2(n_1+n_2)}(V \otimes_\Delta W)_+,$$ it follows that $$\begin{aligned} \| \begin{pmatrix} z_1&0\\0&z_2 \end{pmatrix} \|_{\Delta} & \le \max \{ |\begin{pmatrix} u_1&0\\0&u_2\end{pmatrix}|_{\Delta}, |\begin{pmatrix} u_1'&0 \\ 0&u_2' \end{pmatrix}|_{\Delta} \} \\ & \le \max \{ |u_1|_\Delta, |u_2|_\Delta, |u'_1|_\Delta, |u'_2|_\Delta \} \\ & < \max \{ \|z_1\|_{\Delta}, \|z_2\|_{\Delta} \} + \varepsilon.\end{aligned}$$ Let $z \in M_n(V \otimes W)$ and $\alpha, \beta \in M_{m,n}$. There exist $u,u' \in M_n(V \otimes_\Delta W)_+$ such that $$\begin{pmatrix} u & z \\ z^* & u' \end{pmatrix} \in M_{2n}(V \otimes_\Delta W)_+ \qquad \text{and} \qquad |u|_{\Delta}, |u'|_{\Delta} < \|z\|_{\Delta}+\varepsilon.$$ We may assume that $\|\alpha\|=\|\beta\|$ by homogeneity. Since $$\begin{pmatrix} \alpha u \alpha^* & \alpha z \beta \\ \beta^* z^* \alpha^* & \beta^* u' \beta \end{pmatrix} = \begin{pmatrix} \alpha & 0 \\ 0 & \beta^* \end{pmatrix} \begin{pmatrix} u & z \\ z^* & u' \end{pmatrix} \begin{pmatrix} \alpha^* & 0 \\ 0 & \beta \end{pmatrix} \in M_{2n}(V \otimes_{\Delta} W)_+,$$ we have $$\begin{aligned} \|\alpha z \beta\|_{\Delta} & \le \max \{ |\alpha u \alpha^*|_{\Delta}, |\beta^* u' \beta|_{\Delta} \} \\ & \le \max \{ \|\alpha\|^2 |u|_{\Delta}, \|\beta\|^2|u'|_{\Delta} \} \\ & < \|\alpha \| \|\beta\| (\|z\|_{\Delta} + \varepsilon).\end{aligned}$$ Hence, $(V \otimes W, \{ \|\cdot\|_{M_n(V \otimes_{\Delta} W)} \}_{n \in \mathbb N})$ is an operator space.

Let $v \in M_k(V)$ and $w \in M_l(W)$ with $\|v\|<1, \|w\|<1$. There exist $v_1, v_2 \in M_k(V)^+_{\|\cdot\|<1}$ and $w_1, w_2 \in M_l(W)^+_{\|\cdot\|<1}$ such that $$\begin{pmatrix} v_1 & v \\ v^* & v_2 \end{pmatrix} \in M_{2k}(V)^+ \qquad \text{and} \qquad \begin{pmatrix} w_1 & w \\ w^* & w_2 \end{pmatrix} \in M_{2l}(W)^+.$$ From $$\begin{aligned} \begin{pmatrix} v_1 \otimes w_1 & v \otimes w \\ v^* \otimes w^* & v_2 \otimes w_2 \end{pmatrix} & = \begin{pmatrix} 1&0&0&0 \\ 0&0&0&1 \end{pmatrix} \begin{pmatrix} v_1 & v \\ v^* & v_2 \end{pmatrix} \otimes \begin{pmatrix} w_1 & w \\ w^* & w_2 \end{pmatrix} \begin{pmatrix} 1&0 \\ 0&0 \\ 0&0 \\ 0&1 \end{pmatrix} \\ & \in M_{2kl}(V \otimes_{\Delta} W)_+,\end{aligned}$$
it follows that $$\|v \otimes w\|_{\Delta} \le \max \{ |v_1 \otimes w_1|_{\Delta}, |v_2 \otimes w_2|_{\Delta} \} \le \max \{ \|v_1\| \|w_1\|, \|v_2\| \|w_2\| \} <1.$$ Therefore, $\|\cdot\|_\Delta$ is a subcross matrix norm.

If $u \in M_n(V \otimes_{\Delta} W)_+$, then $$\begin{pmatrix} u&u \\ u&u \end{pmatrix} = \begin{pmatrix} 1 \\ 1 \end{pmatrix} u \begin{pmatrix} 1&1 \end{pmatrix} \in M_{2n}(V \otimes_\Delta W)_+,$$ and consequently, $\|u\|_{\Delta} \le |u|_{\Delta}$. Let $\|z\|_{\Delta}<1$. There exist $u, u' \in M_n(V \otimes_{\Delta} W)_+$ such that $$\begin{pmatrix} u &z \\ z^* & u' \end{pmatrix} \in M_{2n}(V \otimes_\Delta W)_+ \qquad \text{and} \qquad |u|_{\Delta}, |u'|_{\Delta}<1.$$ The matrix regularity follows from $$\|u\|_{\Delta} \le |u|_{\Delta} <1 \qquad \text{and} \qquad \|u'\|_{\Delta} \le |u'|_{\Delta} < 1.$$
\end{proof}

\begin{thm}\label{delta}
Suppose that $V$ and $W$ are matrix regular operator spaces. Then, $$(V \otimes W, \{ \|\cdot\|_{M_n(V \otimes_{\delta} W)} \}_{n \in \mathbb N}, \{ M_n(V \otimes_{\delta} W)^+ \}_{n \in \mathbb N})$$ is a matrix regular operator space with a subcross matrix norm.
\end{thm}

\begin{proof}
Let $z_i \in M_{n_i}(V \otimes W)$ ($i=1,2$). There exist $u_i, u_i' \in M_{n_i}(V \otimes W)$ such that $$\begin{pmatrix} u_i&z_i \\ z_i^*&u_i' \end{pmatrix} \in M_{2n_i}(V \otimes_\delta W)_+ \qquad \text{and} \qquad |u_i|_\delta, |u_i'|_{\delta} < \|z_i\|_{\delta}+\varepsilon \qquad (i=1,2).$$ For c.c.p. maps $\varphi : V \to M_k$ and $\psi : W \to M_l$, the matrix $$(\varphi \otimes \psi)_{2(n_1+n_2)}(\begin{pmatrix} u_1&0&z_1&0 \\ 0&u_2&0&z_2 \\ z_1^*&0&u_1'&0 \\ 0&z_2^*&0&u_2' \end{pmatrix})$$ is unitarily equivalent to $$\begin{pmatrix} (\varphi \otimes \psi)_{n_1} (u_1)&(\varphi \otimes \psi)_{n_1} (z_1)&0&0 \\ (\varphi \otimes \psi)_{n_1} (z_1^*)&(\varphi \otimes \psi)_{n_1} (u_1') &0&0\\ 0&0&(\varphi \otimes \psi)_{n_2} (u_2)&(\varphi \otimes \psi)_{n_2} (z_2) \\ 0&0&(\varphi \otimes \psi)_{n_2} (z_2^*)&(\varphi \otimes \psi)_{n_2} (u_2') \end{pmatrix} \in M_{2(n_1+n_2)kl}^+.$$ Because $$\begin{aligned} \|(\varphi \otimes \psi)_{n_1+n_2}(\begin{pmatrix} u_1&0 \\ 0&u_2\end{pmatrix})\| & =\|\begin{pmatrix} (\varphi \otimes \psi)_{n_1}(u_1) & 0 \\ 0 & (\varphi \otimes \psi)_{n_2}(u_2) \end{pmatrix} \| \\ & \le \max \{ |u_1|_{\delta}, |u_2|_{\delta} \}\\ & < \max \{ \|z_1\|_{\delta}, \|z_2\|_{\delta} \} +\varepsilon, \end{aligned}$$ we see that $$\|\begin{pmatrix} z_1&0 \\ 0&z_2 \end{pmatrix} \|_{\delta} \le \max \{ \|z_1\|_{\delta}, \|z_2\|_{\delta} \}.$$ The remaining proof is similar to that of Theorem \ref{Delta}.
\end{proof}

\begin{thm}
Suppose that $V, W$ and $V \otimes_{\alpha} W$ are matrix regular operator spaces.
\begin{enumerate}
\item If $\|\cdot\|_{\alpha}$ is a subcross matrix norm and $v \otimes w \in M_{kl}(V \otimes_{\alpha} W)^+$ for all $v \in M_k(V)^+, w \in M_l(W)^+$, then we have $$M_n(V \otimes_{\Delta} W)^+ \subset M_n(V \otimes_{\alpha} W)^+ \qquad \text{and} \qquad \|\cdot\|_{\alpha} \le \|\cdot\|_{\Delta}.$$
\item If $\varphi \otimes \psi : V \otimes_{\alpha} W \to M_{kl}$ is c.c.p. for all c.c.p. maps $\varphi : V \to M_k, \psi : W \to M_l$, then we have $$M_n(V \otimes_{\alpha} W)^+ \subset M_n(V \otimes_{\delta} W)^+ \qquad \text{and} \qquad \|\cdot\|_{\delta} \le \|\cdot\|_{\alpha}.$$
\end{enumerate}
\end{thm}

\begin{proof}
(1) Let $u = \alpha (v \otimes w) \alpha^* \in M_n(V \otimes_\Delta W)_+$ for $v \in M_k(V)^+, w \in M_l(W)^+, \alpha \in M_{n,kl}$. Since $V \otimes_\alpha W$ is matrix ordered and $\|\cdot\|_\alpha$ is a subcross matrix norm, it follows that $$u = \alpha (v \otimes w) \alpha^* \in M_n(V \otimes_\alpha W)^+ \quad \text{and} \quad \|u\|_\alpha \le \|\alpha\|^2 \|v\| \|w\|,$$ and hence, $\|u\|_\alpha \le |u|_\Delta$. For a general element $z$ in $M_n(V \otimes W)$, we have $$\begin{aligned} \|z\|_\alpha & = \inf \{ \max \{ \|u\|_\alpha, \|u'\|_\alpha \} : \begin{pmatrix} u&z \\ z^*&u' \end{pmatrix} \in M_{2n}(V \otimes_\alpha W)^+ \} \\ & \le \inf \{ \max \{ |u|_\Delta, |u'|_\Delta \} : \begin{pmatrix} u&z \\ z^*&u' \end{pmatrix} \in M_{2n}(V \otimes_\Delta W)_+ \} \\ & = \|z\|_\Delta. \end{aligned}$$ From $M_n(V \otimes_\Delta W)_+ \subset M_n(V \otimes_{\alpha} W)^+$ and $\|\cdot\|_{\alpha} \le \|\cdot\|_{\Delta}$, we see that $$M_n(V \otimes_{\Delta} W)^+ \subset M_n(V \otimes_{\alpha} W)^+.$$

(2) Let $u \in M_n(V \otimes_\alpha W)^+$. Since $(\varphi \otimes \psi)_n(u) \in M_{nkl}^+$ and $\|(\varphi \otimes \psi)_n(u)\| \le \|u\|_\alpha$, we see that $$M_n(V \otimes_\alpha W)^+ \subset M_n(V \otimes_\delta W)_+=M_n(V \otimes_\delta W)^+ \qquad \text{and} \qquad |u|_\delta \le \|u\|_\alpha.$$ For a general element $z \in M_n(V \otimes W)$, we have $$\begin{aligned} \|z\|_\delta & = \inf \{ \max \{ |u|_\delta, |u'|_\delta \} : \begin{pmatrix} u&z \\ z^*&u' \end{pmatrix} \in M_{2n}(V \otimes_\delta W)_+ \} \\ & \le \inf \{ \max \{ \|u\|_\alpha, \|u'\|_\alpha \} : \begin{pmatrix} u&z \\ z^*&u' \end{pmatrix} \in M_{2n}(V \otimes_\alpha W)^+ \} \\ & = \|z\|_\alpha. \end{aligned}$$
\end{proof}

\begin{prop} \label{linearization}
Suppose that $V, W$ and $Z$ are matrix regular operator spaces and $\Phi : V \times W \to Z$ is a completely bounded bilinear map. Then $\Phi : V \times W \to Z$ is completely positive if and only if its linearization $\tilde{\Phi} : V \otimes_{\Delta} W \to Z$ is completely positive. In this case, $\|\Phi\|_{cb}=\|\tilde{\Phi}\|_{cb}$.
\end{prop}

\begin{proof}
Suppose that the bilinear map $\Phi : V \times W \to Z$ is completely positive. Let $u=\alpha (v \otimes w) \alpha^*$ for $v \in M_p(V)^+, w \in M_q(W)^+, \alpha \in M_{n,pq}$. Since $$\tilde{\Phi}_n (u) = \tilde{\Phi}_n(\alpha (v \otimes w) \alpha^*) = \alpha \Phi_{pq}(v,w)\alpha^* \in M_n(Z)^+,$$ $\tilde{\Phi}_n(M_n(V \otimes_\Delta W)_+) \subset Z^+$ and $\|\tilde{\Phi}_n(u)\| \le \|\Phi\|_{cb} |u|_{\Delta}$

Choose an element $z$ in $M_n(V \otimes_{\Delta} W)_{\|\cdot\|<1}$. There exist $u, u' \in M_n(V \otimes_{\Delta} W)_+$ such that $$|u|_{\Delta}, |u'|_{\Delta}<1 \qquad \text{and} \qquad \begin{pmatrix} u&z \\ z^*&u' \end{pmatrix} \in M_{2n}(V \otimes_{\Delta} W)_+.$$ Applying $\tilde{\Phi}_{2n}$ to the above matrix, we get $$\begin{pmatrix} \tilde{\Phi}_n(u) & \tilde{\Phi}_n(z) \\ \tilde{\Phi}_n(z^*) & \tilde{\Phi}_n(u') \end{pmatrix} = {\tilde \Phi}_{2n} (\begin{pmatrix} u&z \\ z^*&u' \end{pmatrix}) \in M_{2n}(Z)^+.$$ By the matrix regularity of $Z$, we have $$\|\tilde{\Phi}_n(z)\| \le \max \{ \|\tilde{\Phi}_n(u)\|, \|\tilde{\Phi}_n(u')\| \} \le \|\Phi\|_{cb} \max \{ |u|_{\Delta}, |u'|_{\Delta} \} < \|\Phi\|_{cb}.$$ Since $\|\cdot\|_{\Delta}$ is a subcross matrix norm, we see that $\|\Phi\|_{cb}=\|\tilde \Phi\|_{cb}$. The complete positivity of $\tilde \Phi$ follows from its continuity.
\end{proof}

\begin{prop}\label{cp maps}
Suppose that $V_1, V_2$ and $W$ are matrix regular operator spaces and $\Phi : V_1 \to V_2$ is completely positive.
\begin{enumerate}
\item The map $$\Phi \otimes {\rm id}_W : V_1 \otimes_{\Delta} W \to V_2 \otimes_{\Delta} W$$ is completely positive with  $\|\Phi \otimes {\rm id}_W : V_1 \otimes_{\Delta} W \to V_2 \otimes_{\Delta} W \|_{cb} \le \|\Phi\|_{cb}$.
\item The map $$\Phi \otimes {\rm id}_W : V_1 \otimes_{\delta} W \to V_2 \otimes_{\delta} W$$ is completely positive with  $\|\Phi \otimes {\rm id}_W : V_1 \otimes_{\delta} W \to V_2 \otimes_{\delta} W \|_{cb} \le \|\Phi\|_{cb}$.
\end{enumerate}
\end{prop}

\begin{proof}
(1) Since $$(\Phi \otimes {\rm id}_W)_n (\alpha (v \otimes w) \alpha^*) = \alpha \Phi_k (v) \otimes w \alpha^*,$$ we have $(\Phi \otimes {\rm id}_W)_n (M_n(V_1 \otimes_\Delta W)_+) \subset M_n(V_2 \otimes_\Delta W)_+$ and $|(\Phi \otimes {\rm id}_W)_n (u)|_\Delta \le \|\Phi\|_{cb} |u|_\Delta$ for $u \in M_n(V_1 \otimes_\Delta W)_+$. Let $z \in M_n(V_1 \otimes W)$ with $\|z\|_\Delta <1$. There exist $u, u' \in M_n(V_1 \otimes_{\Delta} W)_+$ such that $$\begin{pmatrix} u &z \\ z^* & u' \end{pmatrix} \in M_{2n}(V_1 \otimes_\Delta W)_+ \qquad \text{and} \qquad |u|_{\Delta}, |u'|_{\Delta}<1.$$ The application of $(\Phi \otimes {\rm id}_W)_{2n}$ yields $$\begin{pmatrix} (\Phi \otimes {\rm id}_W)_n (u) &(\Phi \otimes {\rm id}_W)_n (z) \\ (\Phi \otimes {\rm id}_W)_n (z)^* & (\Phi \otimes {\rm id}_W)_n (u') \end{pmatrix} \in M_{2n}(V_2 \otimes_\Delta W)_+.$$ It follows that $$\|(\Phi \otimes {\rm id}_W)_n (z)\|_\Delta \le \max \{ |(\Phi \otimes {\rm id}_W)_n (u)|_\Delta, |(\Phi \otimes {\rm id}_W)_n (u')|_\Delta \} < \|\Phi\|_{cb}.$$ Complete positivity is obtained by the continuity.

(2) Let $\varphi : V_2 \to M_k, \psi : W \to M_l$ be c.c.p. maps and $u \in M_n(V_1 \otimes_{\delta} W)^+$. Since $$(\varphi \otimes \psi)_n \circ (\Phi \otimes {\rm id}_W)_n (u) = \|\Phi\|_{cb}(\varphi \circ (\Phi \slash \|\Phi\|_{cb}) \otimes \psi)_n (u),$$ $\Phi \otimes {\rm id}_W$ is completely positive and $|(\Phi \otimes {\rm id}_W)_n (u)|_\delta \le \|\Phi\|_{cb} |u|_\delta$ for $u \in M_n(V_1 \otimes_\delta W)_+$. The remaining proof is similar to (1).

\end{proof}

\section{Decomposable maps}

We define decomposable maps between matrix regular operator spaces and their decomposable norms in a similar manner as for those for $C^*$-algebras.

\begin{defn}
Suppose that $V$ and $W$ are matrix regular operator spaces
\begin{enumerate}
\item A linear map $T : V \to W$ is called decomposable if it is a linear combination of completely positive maps.
\item For a decomposable map $T : V \to W$, we define its decomposable norm $\|T\|_{dec}$ by $$\|T\|_{dec} = \inf \{ \max \{ \|S_1\|, \|S_2\|\} : \begin{pmatrix} S_1 & T \\ T^* & S_2 \end{pmatrix} \ge_{cp} 0 \}.$$
\end{enumerate}
\end{defn}

A linear map $T : V \to W$ is decomposable if and only if there exist two completely positive maps $S_1, S_2 : V \to W$ for which the linear map $$\Phi : V \to M_2(W) : v \mapsto \begin{pmatrix} S_1(v) & T(v) \\ T^*(v) & S_2(v) \end{pmatrix}$$ is completely positive, whose proof is same with \cite{H}. The space of decomposable maps from $V$ into $W$ equipped with decomposable norm $\|\cdot\|_{dec}$ is denoted by $D(V, W)$.

\begin{prop}
Suppose that $V_1, V_2$ and $W$ are matrix regular operator spaces and $\Phi : V_1 \to V_2$ is decomposable.
\begin{enumerate}
\item The map $$\Phi \otimes {\rm id}_W : V_1 \otimes_{\Delta} W \to V_2 \otimes_{\Delta} W$$ is decomposable with  $\|\Phi \otimes {\rm id}_W\|_{dec} \le \|\Phi\|_{dec}$.
\item The map $$\Phi \otimes {\rm id}_W : V_1 \otimes_{\delta} W \to V_2 \otimes_{\delta} W$$ is decomposable with  $\|\Phi \otimes {\rm id}_W \|_{dec} \le \|\Phi\|_{dec}$.
\end{enumerate}
\end{prop}

\begin{proof}
(1) There exist completely positive maps $\Psi : V_1 \to V_2$ and $\Psi' : V_1 \to V_2$ such that the map $$\begin{pmatrix} \Psi & \Phi \\ \Phi^* & \Psi' \end{pmatrix} : V_1 \to M_2(V_2)$$ is completely positive and $\|\Psi\|_{cb}, \|\Psi'\|_{cb} < \|\Phi\|_{dec}+\varepsilon$. Let $v \in M_k(V_1)^+, w \in M_l(W)^+$ and $\alpha \in M_{n,kl}$. By the canonical shuffle \cite{Pa2}, the matrix $(\begin{pmatrix} \Psi \otimes {\rm id}_W &\Phi \otimes {\rm id}_W \\ \Phi^* \otimes {\rm id}_W &\Psi' \otimes {\rm id}_W \end{pmatrix})_n(\alpha (v \otimes w) \alpha^*)$ is unitarily equivalent to $$\begin{aligned} & \begin{pmatrix} \alpha \Psi_k(v) \otimes w \alpha^* & \alpha \Phi_k (v) \otimes w \alpha^* \\ \alpha \Phi^*_k(v) \otimes w \alpha^* & \alpha \Psi'_k (v) \otimes w \alpha^* \end{pmatrix} \\ = & \begin{pmatrix} \alpha & 0 \\ 0 & \alpha \end{pmatrix} (\begin{pmatrix} \Psi_k(v) & \Phi_k(v) \\ \Phi^*_k(v) & \Psi'_k(v) \end{pmatrix} \otimes w) \begin{pmatrix} \alpha^* & 0 \\ 0 & \alpha^* \end{pmatrix} \in M_{2n}(V_2 \otimes_{\Delta} W)^+. \end{aligned}$$ Therefore, the map $$\begin{pmatrix} \Psi \otimes {\rm id}_W & \Phi \otimes {\rm id}_W \\ (\Phi \otimes {\rm id}_W)^* & \Psi' \otimes {\rm id}_W \end{pmatrix} : V_1 \otimes_{\Delta} W \to M_2(V_2 \otimes_{\Delta} W)$$ is completely positive and $$\|\Phi \otimes {\rm id}_W \|_{dec} \le \max \{ \|\Psi \otimes {\rm id}_W\|_{cb}, \|\Psi' \otimes {\rm id}_W\|_{cb}\} \le \max \{ \|\Psi\|_{cb}, \|\Psi'\|_{cb} \} < \|\Phi\|_{dec}+\varepsilon$$ by Proposition \ref{cp maps} (1).

(2) Let $\varphi : V_2 \to M_k, \psi : W \to M_l$ be c.c.p. maps and $u \in M_n(V_1 \otimes_{\delta} W)^+$. Since $$(\varphi \otimes \psi)_{2n}((\begin{pmatrix} \Psi \otimes {\rm id}_W &\Phi \otimes {\rm id}_W \\ \Phi^* \otimes {\rm id}_W &\Psi' \otimes {\rm id}_W \end{pmatrix})_n(u)) = (\begin{pmatrix} \varphi \circ \Psi & \varphi \circ \Phi \\ \varphi \circ \Phi^* & \varphi \circ \Psi' \end{pmatrix} \otimes \psi)_n(u) \in M_{2nkl}^+,$$ the map $$\begin{pmatrix} \Psi \otimes {\rm id}_W & \Phi \otimes {\rm id}_W \\ (\Phi \otimes {\rm id}_W)^* & \Psi' \otimes {\rm id}_W \end{pmatrix} : V_1 \otimes_{\delta} W \to M_2(V_2 \otimes_{\delta} W)$$ is completely positive and $$\|\Phi \otimes {\rm id}_W \|_{dec} \le \max \{ \|\Psi \otimes {\rm id}_W\|_{cb}, \|\Psi' \otimes {\rm id}_W\|_{cb}\} \le \max \{ \|\Psi\|_{cb}, \|\Psi'\|_{cb} \} < \|\Phi\|_{dec}+\varepsilon$$
by Proposition \ref{cp maps} (2).
\end{proof}

As done for the operator space structure on the space of completely bounded maps and \cite[Section 6]{LMM}, we define the matrix norms and positive cones on each matrix level of $D(V,W)$ by $$M_n(D(V,W)=D(V,M_n(W)) \qquad \text{and} \qquad M_n(D(V,W))^+ = CP(V,M_n(W)),$$ respectively.

\begin{thm}\label{duality1}
Suppose that $V$ and $W$ are matrix regular operator spaces. The canonical map $$\Phi : (V \otimes_{\Delta} W)^* \to D(V,W^*)$$ is a completely isometric and completely order isomorphic isomorphism.
\end{thm}

\begin{proof}
Let $\varphi \in M_n((V \otimes_\Delta W)^*)=CB(V \otimes_\Delta W, M_n)$. Since $$((\Phi_n (\varphi)_k(v))_l(w) = \varphi_{kl}(v \otimes w), \qquad v \in M_k(V), w \in M_l(W),$$ $\varphi : V \otimes_\Delta W \to M_n$ is completely positive if and only if $\Phi_n (\varphi) : V \to M_n(W^*)$ is completely positive, in this case, $\|\varphi\|_{cb} = \|\Phi_n (\varphi)\|_{cb}$ by Proposition \ref{linearization}.

Next, let $\varphi \in M_n((V \otimes_\Delta W)^*)_1$. Since the dual spaces of matrix regular operator spaces are matrix regular \cite[Corollary 4.7]{S}, there exist $\psi_1, \psi_2 \in M_n((V \otimes_{\Delta} W)^*)^+_1$ such that $$\begin{pmatrix} \psi_1 & \varphi \\ \varphi^* & \psi_2 \end{pmatrix} \in M_{2n}((V \otimes_{\Delta} W)^*)^+.$$ Applying $\Phi_{2n}$, we get $$\begin{pmatrix} \Phi_n(\psi_1) & \Phi_n(\varphi) \\ \Phi_n(\varphi)^* & \Phi_n(\psi_2) \end{pmatrix} \in CP(V, M_{2n}(W^*)) \quad \text{and} \quad \Phi_n(\psi_i) \in CCP(V,M_n(W^*)) \quad (i=1,2).$$ Hence, $\|\Phi_n(\varphi)\|_{dec} \le 1$. The converse is merely the reverse of the above argument.
\end{proof}

Let $\mathcal A$ be a $C^*$-algebra and $\mathcal R$ a von Neumann algebra. Recall that the normed space dual of the projective tensor product $\mathcal A \otimes_\gamma \mathcal R_*$ is the space $B(\mathcal A, \mathcal R)$ of bounded linear maps and the operator space dual of the operator space projective tensor product $\mathcal A \otimes_\wedge \mathcal R_*$ is the space $CB(\mathcal A, \mathcal R)$ of completely bounded linear maps. Now, we see that the matrix regular operator space dual of $\mathcal A \otimes_\Delta \mathcal R_*$ is the space $D(\mathcal A, \mathcal R)$ of decomposable maps.

\begin{cor}
Let $\mathcal A$ be a $C^*$-algebra and $\mathcal R$ a von Neumann algebra. Then the matrix regular operator space $\mathcal A \otimes_{\Delta} \mathcal R_*$ is the predual of space ${\mathcal D}({\mathcal A},{\mathcal R})$ of Haagerup's decomposable maps from $\mathcal A$ into $\mathcal R$.
\end{cor}

\begin{lem}\label{positive}
Suppose that $z$ is a positive element in $M_n(V \otimes_\delta W)$. Then we have $$|z|_{\delta}=\|z\|_\delta = \|z\|_\vee.$$
\end{lem}

\begin{proof}
(1) Since $$\begin{pmatrix} z&z \\ z&z \end{pmatrix} = \begin{pmatrix} 1 \\ 1 \end{pmatrix} z \begin{pmatrix} 1&1 \end{pmatrix} \in M_{2n}(V \otimes_\delta W)_+,$$ we have $\|z\|_{\delta} \le |z|_{\delta}$. For the converse, we take $u, u'$ in $M_n(V \otimes_{\delta} W)^+$ such that $$\begin{pmatrix} u&z \\ z^*&u' \end{pmatrix} \in M_{2n}(V \otimes_{\delta} W)^+ \qquad \text{and} \qquad |u|_{\delta}, |u'|_{\delta} < \|z\|_\delta + \varepsilon.$$ For c.c.p. maps $\varphi : V \to M_k$ and $\psi : W \to M_l$, we have $$\begin{pmatrix} (\varphi \otimes \psi)_n(u) & (\varphi \otimes \psi)_n(z) \\ (\varphi \otimes \psi)_n(z)^* & (\varphi \otimes \psi)_n(u') \end{pmatrix} \in M_{2nkl}^+.$$ It follows that $$\|(\varphi \otimes \psi)_n(z)\| \le \max \{ \|(\varphi \otimes \psi)_n(u)\|, \|(\varphi \otimes \psi)_n(u') \| \} \le \max \{ |u|_\delta, |u'|_\delta \} < \|z\|_\delta + \varepsilon,$$ so, $|z|_\delta \le \|z\|_\delta$.

(2) From the definitions, it is obvious that $|z|_\delta \le \|z\|_\vee$. For the converse, we take complete contractions $\varphi : V \to M_k$ and $\psi : W \to M_l$. We can regard $\varphi$ and $\psi$ as elements in $M_k(V^*)_1$ and $M_l(W^*)_1$, respectively. There exist $\varphi_i \in M_k(V^*)^+_1$ and $\psi_i \in M_l(W^*)^+_1$ ($i=1,2$) such that $$\begin{pmatrix} \varphi_1&\varphi \\ \varphi^*&\varphi_2 \end{pmatrix} \in M_{2k}(V^*)^+ \qquad \text{and} \qquad \begin{pmatrix} \psi_1&\psi \\ \psi^*&\psi_2 \end{pmatrix} \in M_{2k}(W^*)^+.$$ Because $$\begin{aligned} \begin{pmatrix} (\varphi_1 \otimes \psi_1)_n (z) & (\varphi \otimes \psi)_n(z) \\ (\varphi \otimes \psi)_n(z)^* & (\varphi_2 \otimes \psi_2)_n(z) \end{pmatrix} & = \begin{pmatrix} 1&0&0&0 \\ 0&0&0&1 \end{pmatrix} (\begin{pmatrix} \varphi_1&\varphi \\ \varphi^*&\varphi_1 \end{pmatrix} \otimes \begin{pmatrix} \psi_1&\psi \\ \psi^*&\psi_1 \end{pmatrix})_n (z) \begin{pmatrix} 1&0\\0&0\\0&0\\0&1 \end{pmatrix} \\ & \in M_{2nkl}^+, \end{aligned}$$ we get $$\|(\varphi \otimes \psi)_n(z)\| \le \max \{ \|(\varphi_1 \otimes \psi_1)_n(z)\|, \|(\varphi_2 \otimes \psi_2)_n(z)\| \} \le |z|_\delta.$$ Hence, $\|z\|_\vee \le |z|_\delta$.
\end{proof}

\begin{thm}\label{duality2}
Suppose that $V$ and $W$ are matrix regular operator spaces. The canonical map $$\Psi : V^* \otimes_{\delta} W \to D(V,W)$$ is a completely contractive and completely order isomorphic injection. If either $V$ or $W$ is finite dimensional, then $\Psi$ is a completely isometric isomorphism.
\end{thm}

\begin{proof}
For $z \in M_n(V^* \otimes W), v \in M_k(V)$ and a linear map $\psi : W \to M_l$, the relation $$\psi_{nk} ((\Psi_n(z))_k(v)) = (\hat{v} \otimes \psi)_n(z) \in M_{nkl}$$ holds. Since the canonical inclusion $V \hookrightarrow V^{**}$ is completely order isomorphic and $M_k(V)^+$ is weak$^*$-dense in $M_k(V^{**})^+$, we see that $\Psi_n(z) : V \to M_n(W)$ is completely positive if and only if $z \in M_n(V^* \otimes_{\delta} W)^+$.

Let $z \in M_n(V^* \otimes W)$ with $\|z\|_{\delta}<1$. Then there exist elements $u, u'$ in $M_n(V^* \otimes_{\delta} W)_+$ such that $$\begin{pmatrix} u&z \\ z^*&u' \end{pmatrix} \in M_{2n}(V^* \otimes_{\delta} W)_+ \qquad \text{and} \qquad |u|_{\delta}, |u'|_{\delta}<1.$$ The linear map $$\begin{pmatrix} \Psi_n (u)&\Psi_n (z) \\ \Psi_n (z)^*&\Psi_n(u') \end{pmatrix} = \Psi_{2n} (\begin{pmatrix} u&z \\ z^*&u' \end{pmatrix}) : V \to M_{2n}(W)$$ is completely positive. By Lemma \ref{positive}, we have $$\|\Psi_n(u)\|_{cb}=\|u\|_\vee = |u|_{\delta}<1 \qquad \text{and} \qquad \|\Psi_n (u')\|_{cb}=\|u'\|_\vee = |u'|_{\delta}<1,$$ so, $\|\Psi_n(z)\|_{dec}<1$. Hence, $\Psi$ is completely contractive.

If either $V$ or $W$ is finite dimensional, then every linear map from $V$ into $W$ has a finite rank. Consequently, we can find $u$ and $u'$ in the reverse of the above argument.
\end{proof}

\begin{cor}\label{duality}
Suppose that $V$ and $W$ are matrix regular operator spaces such that one of them is finite dimensional. Then the identification $$V^* \otimes_{\delta} W^* = (V \otimes_{\Delta} W)^*$$ is completely isometric and completely order isomorphic.
\end{cor}

\begin{proof}
By Theorems \ref{duality1} and \ref{duality2}, the identifications $$V^* \otimes_{\delta} W^* = D (V,W^*)=(V \otimes_{\Delta} W)^*$$ are completely isometric and completely order isomorphic.
\end{proof}

The following corollary is the reason why we chose the symbol $\delta$.

\begin{cor}\label{pisier}
Suppose that $\mathcal A$ and $\mathcal B$ are $C^*$-algebras. The canonical map $$\Psi : \mathcal A^* \otimes_{\delta} \mathcal B \to D(\mathcal A,\mathcal B)$$ is a completely isometric and completely order isomorphic injection. Consequently, the norm $\|\cdot\|_{\delta}$ on $\mathcal A^* \otimes \mathcal B$ coincides with Pisier's delta norm.
\end{cor}

\begin{proof}
By Theorem \ref{duality2}, $\Psi$ is completely contractive and completely order isomorphic. Let $\|\Psi_n(z)\|_{dec}<1$. Since $\Psi_n(z) : \mathcal A \to M_n(\mathcal B)$ has a finite rank, it has a factorization
$$\xymatrix{\mathcal A \ar[rr]^{\Psi_n(z)} \ar[dr]_S && M_n(\mathcal B) \\ & M_k \ar[ur]_T &}$$ with $\|S\|_{dec}<1, \|T\|_{dec}<1$ by \cite[Theorem 2.1]{JLM}. There exist c.p. maps $S_i : \mathcal A \to M_k$ and $T_i : M_k \to M_n(\mathcal B)$ ($i=1,2$) such that $\|S_i\|_{cb}, \|T_i\|_{cb} <1$ and $$\begin{pmatrix} S_1 & S \\ S^* & S_2 \end{pmatrix}, \begin{pmatrix} T_1 & T \\ T^* & T_2 \end{pmatrix} \ge_{cp} 0.$$ Since $S_i : \mathcal A \to M_k$ and $T_i : M_k \to M_n(\mathcal B)$ have finite ranks, we can find $u_i \in M_n(\mathcal A^* \otimes \mathcal B)$ such that $\Psi_n(u_i)=T_i \circ S_i$. By \cite[Remark 1.2]{H}, the map $$\Psi_{2n}(\begin{pmatrix} u_1 & z \\ z^* & u_2 \end{pmatrix}) = \begin{pmatrix} T_1 \circ S_1 & \Psi_n(z) \\ \Psi_n(z)^* & T_2 \circ S_2 \end{pmatrix} : \mathcal A \to M_{2n}(\mathcal B)$$ is completely positive. Since $|u_i|_\delta = \|u_i\|_\vee = \|T_i \circ S_i\|_{cb}<1$, we have $\|z\|_\delta <1$.

The last statement follows from \cite[Corollary 12.5]{Pi} and \cite[Theorem 2.1]{JLM}.
\end{proof}

Haagerup's results \cite[Theorem 2.1]{H} and \cite[Proposition 3.4]{H} can be rephrased as the following two corollaries.

\begin{cor}
Let $\mathcal R$ be a von Neumann algebra. Then the following seven conditions are equivalent:
\begin{enumerate}
\item[(i)] $\mathcal R$ is injective;
\item[(ii)] for every $C^*$-algebra $\mathcal A$, two norms $\|\cdot\|_\wedge$ and $\|\cdot\|_\Delta$ on $\mathcal A \otimes \mathcal R_*$ coincide;
\item[(iii)] for every $C^*$-algebra $\mathcal A$, two norms $\|\cdot\|_\vee$ and $\|\cdot\|_\delta$ on $\mathcal A^* \otimes \mathcal R$ coincide;
\item[(iv)] for every $n \in \mathbb N$, two norms $\|\cdot\|_\wedge$ and $\|\cdot\|_\Delta$ on $\ell_\infty^n \otimes \mathcal R_*$ coincide;
\item[(v)] for every $n \in \mathbb N$, two norms $\|\cdot\|_\vee$ and $\|\cdot\|_\delta$ on $\ell_1^n \otimes \mathcal R$ coincide;
\item[(vi)] two norms $\|\cdot\|_\wedge$ and $\|\cdot\|_\Delta$ on $\ell_\infty^n \otimes \mathcal R_*$ are equivalent uniformly for $n \in \mathbb N$;
\item[(vii)] two norms $\|\cdot\|_\vee$ and $\|\cdot\|_\delta$ on $\ell_1^n \otimes \mathcal R$ are equivalent uniformly for $n \in \mathbb N$.
\end{enumerate}
\end{cor}

\begin{cor}
Let $\mathcal R$ be a von Neumann algebra. Then,
\begin{enumerate}
\item two norms $\|\cdot\|_\wedge$ and $\|\cdot\|_\Delta$ coincide on $\ell_\infty^2 \otimes \mathcal R_*$.
\item two norms $\|\cdot\|_\vee$ and $\|\cdot\|_\delta$ coincide on $\ell_1^2 \otimes \mathcal R$.
\end{enumerate}
\end{cor}

A finite factor $\mathcal R$ can be regarded as an algebraic subspace of its predual $\mathcal R_*$ via $$x \in \mathcal R \mapsto {\rm tr}(\ \cdot \ x) \in \mathcal R_*.$$ The inclusion $\mathcal R \subset \mathcal R_*$ is order isomorphic.

\begin{thm}
Suppose that $\mathcal R$ is a finite factor and that $u_1,\cdots, u_n \in \mathcal R$ are arbitrary unitaries. Then we have $$\|\sum_{k=1}^n e_k \otimes u_k \|_{\ell_\infty^n \otimes_\Delta \mathcal R_*} = 1.$$
\end{thm}

\begin{proof}
From $$\begin{aligned} \begin{pmatrix} 1 \otimes 1 & \sum_{k=1}^n e_k \otimes u_k \\ \sum_{k=1}^n e_k \otimes u_k^* & 1 \otimes 1 \end{pmatrix} & = \sum_{k=1}^n \begin{pmatrix} e_k \otimes 1 & e_k \otimes u_k \\ e_k \otimes u_k^* & e_k \otimes 1 \end{pmatrix}\\ & = \sum_{k=1}^n e_k \otimes \begin{pmatrix} 1 & u_k \\ u_k^* & 1 \end{pmatrix} \\ & \in M_2(\ell_\infty^n \otimes_\Delta \mathcal R_*)_+, \end{aligned}$$ it follows that $$\|\sum_{k=1}^n e_k \otimes u_k \|_{\ell_\infty^n \otimes_\Delta \mathcal R_*} \le | 1 \otimes 1 |_\Delta \le 1.$$ For the converse, we will appeal to the duality. Let $c$ denote the counting measure. By Theorem \ref{delta} and Corollary \ref{duality}, we have $$\begin{aligned} n & = (c \otimes {\rm tr})(1 \otimes 1) \\ & = (c \otimes {\rm tr})((\sum_{k=1}^n e_k \otimes u_k)(\sum_{l=1}^n e_l \otimes u_l^*)) \\ & \le \|\sum_{k=1}^n e_k \otimes u_k \|_{\ell_\infty^n \otimes_\Delta \mathcal R_*} \|\sum_{l=1}^n e_l \otimes u_l^* \|_{\ell_1^n \otimes_\delta \mathcal R} \\ & \le \|\sum_{k=1}^n e_k \otimes u_k \|_{\ell_\infty^n \otimes_\Delta \mathcal R_*} (\sum_{l=1}^n \|e_l\|_{\ell^n_1} \|u_l^*\|_{\mathcal R})\\ & = n \|\sum_{k=1}^n e_k \otimes u_k \|_{\ell_\infty^n \otimes_\Delta \mathcal R_*}.
\end{aligned}$$
\end{proof}

\begin{cor}
Suppose that $g_1, \cdots, g_n$ are the generators for the free group $\mathbb F_n$. Then we have $$\|\sum_{k=1}^n e_k \otimes \lambda (g_k) \|_{\ell_\infty^n \otimes_\Delta L(\mathbb F_n)_*} = 1 \qquad \text{and} \qquad \|\sum_{k=1}^n e_k \otimes \lambda (g_k) \|_{\ell_\infty^n \otimes_\wedge L(\mathbb F_n)_*} \ge {n \over 2\sqrt{n-1}}$$ for $n \ge 2$.
\end{cor}

\begin{proof}
According to \cite[Theorem 5.4.7]{ER2}, $$\|\sum_{k=1}^n e_k \otimes \lambda (g_k) \|_{\ell_1^n \otimes_\vee L(\mathbb F_n)} \le 2 \sqrt{n-1}.$$ The remaining proof is similar to the above if we apply the duality between the operator space injective and projective tensor products.
\end{proof}

\begin{thm}
Suppose that $V$ and $W$ are matrix regular operator spaces and $z$ is an element of $V^* \otimes W$.
\begin{enumerate}
\item $z$ belongs to $(V^* \otimes_{\Delta} W)_+$ if and only if its associated finite rank map $\Psi (z) : V \to W$ has a factorization $$\Psi (z) = S \circ R$$ where $R : V \to M_k$ and $S : M_k \to W$ are completely positive.
\item Under assumption (1), $|z|_{\Delta}<1$ if and only if we can take $R$ as a c.c.p. map and $S(x)=\alpha x \otimes w \alpha^*$ for a row vector $\alpha$ with $\|\alpha\|_{\ell^{kl}_2} <1$ and $w \in M_l(W)^+_1$.
\item Let $z \in V^* \otimes W$. Then $\|z\|_{\Delta}<1$ if and only if there exist $\Phi_i : V \to W (i=1,2)$ such that the map $$\begin{pmatrix} \Phi_1 & \Psi (z) \\ \Psi (z)^* & \Phi_2 \end{pmatrix} : V \to M_2(W)$$ has a completely positive matrix factorization and $\Phi_i$ has the factorization $\Phi_i = S_i \circ R_i$ such that $R_i : V \to M_{k_i}$ is c.c.p. and $S_i : M_{k_i} \to W$ has the form $S_i(x)=\alpha_i (x \otimes w_i) \alpha_i^*$ for a row vector $\alpha_i$ with $\|\alpha_i\|_{\ell^{k_i l_i}_2}<1$ and $w_i \in M_{l_i}(W)^+_1$.
\end{enumerate}
\end{thm}

\begin{proof}
(1) $\Rightarrow )$ $z$ can be written as $$z=\alpha (R \otimes w) \alpha^*, \qquad \alpha \in M_{1,kl}, R \in M_k(V^*)^+, w \in M_l(W).$$ Since $M_k(V^*)^+=CP(V,M_k)$, we can regard $R$ as a completely positive map from $V$ into $M_k$. We define $S : M_k \to W$ by $S(x)=\alpha (x \otimes w) \alpha^*$. Then $S$ is completely positive and $\Psi(z)=S \circ R$.

$\Leftarrow )$ Since $[e_{ij}] \in M_k(M_k)^+$ and $S : M_k \to W$ is completely positive, $[S(e_{ij})]$ belongs to $M_k(W)^+$. We check the identity $$S(x)=\begin{pmatrix} e_1^t, \cdots, e_n^t \end{pmatrix} ( x \otimes [S(e_{ij})] ) \begin{pmatrix} e_1 \\ \vdots \\ e_n \end{pmatrix}$$ only for $x = e_{kl}$. If we regard $R \in M_k(V^*)^+$, then it can be verified that $$z=\begin{pmatrix} e_1^t, \cdots, e_n^t \end{pmatrix} ( R \otimes [S(e_{ij})] ) \begin{pmatrix} e_1 \\ \vdots \\ e_n \end{pmatrix} \in (V^* \otimes_{\Delta} W)_+.$$

(2)(3) The proofs are left to the reader.
\end{proof}

\bigskip

{\bf Acknowledgments}

The author is grateful to the referee for careful reading and bringing his attention to Ref. \cite{LMM}.


\end{document}